\documentclass[13pt]{amsart}
\usepackage{amsmath,amssymb,amsthm,amscd}
\usepackage{graphicx}
\usepackage{color}
\numberwithin{equation}{section}

\newtheorem{prop}{Proposition}[section]
\newtheorem{theo}[prop]{Theorem}
\newtheorem{lem}[prop]{Lemma}
\newtheorem{coro}[prop]{Corollary}
\newtheorem{rem}[prop]{Remark}

\newtheorem{defi}[prop]{Definition}

\def\begeq{\begin{equation}}
\def\endeq{\end{equation}}

\begin{document}
\title{Finiteness of $\mathbb Q$-Fano compactifications of semisimple group with K\"ahler-Einstein metrics}
\author{Yan Li$^*$ and ZhenYe Li $^{\dag}$}

\address{$^{*}$School of Mathematics and Statistics, Beijing Institute of Technology, Beijing, 100081, China.}
\address{$^\dag$College of Mathematics and Physics, Beijing University of Chemical Technology, Beijing, 100029, China.}
\email{liyan.kitai@yandex.ru,\ \ \ lizhenye@pku.edu.cn}

\thanks {$^*$Partially supported by the China Post-Doctoral Grant BX20180010.}



\begin{abstract}
In this note, we give a way to classify $\mathbb Q$-Fano compactifications of a semisimple group $G$. 
We will prove that there are only finitely many such $\mathbb Q$-Fano $G$-compactifications, which admits (singular) K\"ahler-Einstein metrics. As an application, this improves a former result in \cite{LTZ2}.
\end{abstract}
\maketitle

\section{Introduction}
Let $G$ be an $n$-dimensional connect, complex reductive group which is the complexification of a compact Lie group $K$, with $J$ its complex structure.
Let $M$ be a projective normal variety. $M$ is called a {\it (bi-equivariant) compactification of $G$} (or \emph{$G$-compactification} for simplicity) if it admits a holomorphic $G\times G$-action with an open and dense orbit isomorphic to $G$ as a $G\times G$-homogeneous space.  $(M, L)$ is called a {\it polarized compactification} of $G$   if  $L$ is a $G\times G$-linearized ample line bundle on $M$. In particular, when $K^{-1}_M$ is an ample $\mathbb Q$-Cartier line bundle and $L=K^{-1}_M$, we call $M$ a \emph{$\mathbb Q$-Fano $G$-compactification} (cf. \cite[Section 2.1]{AK} and \cite{AB1, AB2, Timashev-Sbo}). For more knowledge and examples, we refer the reader to \cite{Timashev-Sbo, AK, Del2, Del3}, etc.

Let $(M,K^{-1}_M)$ be a $\mathbb Q$-Fano compactification of $G$. Fix a maximal complex torus $T^{\mathbb C}$ (denote by $r$ its dimension) of $G$. It is known that the closure $Z$ of $T^\mathbb C$ in $M$, together with $K^{-1}_M|_Z$ is a polarized toric variety. Indeed, $K^{-1}_M|_Z$ is $WT^\mathbb C$-linearized, where $W$ is the Weyl group with respect to $G$ and $T^\mathbb C$. The polytope associated to $(M,K^{-1}_M)$ is defined as the associated polytope of $(Z,K_M^{-1}|_Z)$ (cf. \cite[Section 2.1]{AK} and \cite[Section 2.2]{Del2}). It is a strictly convex, $W$-invariant rational polytope in $\mathfrak a^*=J\mathfrak t^*$.
Choose a set of positive roots $\Phi_+$ and denote by $\mathfrak a^*_+$ the corresponding positive Weyl chamber. Choose a $W$-invariant inner product $\langle\cdot,\cdot\rangle$ on $\mathfrak a^*$ which extends the Cartan-Killing form on the semisimple part $\mathfrak a^*_{ss}$ (cf. \cite[Introduction]{Del2}). Let $P_+$ be the positive part of $P$ defined by
$$P_+\,=\,\{y\in P|~ \langle\alpha,y\rangle\geq0, ~\forall \alpha\in \Phi_+\}.$$
Define the weighted barycenter of $P_+$ by
$$\mathbf{b}(P_+)\,=\,\frac{\int_{P_+}y\pi(y) \,dy}{\int_{P_+}\pi(y) \,dy},$$
where $\pi(y)=\prod_{\alpha\in\Phi_+}\langle\alpha,y\rangle^2.$
In \cite{LTZ2}, Li-Tian-Zhu proved the following criterion of existence of (singular) K\"ahler-Einstein metric on a $\mathbb Q$-Fano group compactification:
\begin{theo}\label{sg-bary-thm} Let $M$ be a $\mathbb Q$-Fano $G$-compactification whose associated polytope satisfies the fine condition.\footnote{We use the terminology ``fine" in sense of \cite{Do}, namely, each vertex of $P$ is the intersection of precisely $r$ facets.} Then $M$ admits a K\"ahler-Einstein metric if and only if
\begin{align}\label{sg-bary}
\mathbf{b}(P_+)\,\in\, 2\rho+\Xi,
\end{align}
where $2\rho=\sum_{\alpha\in\Phi_+}\alpha$
and
$\Xi$ is the relative interior of the cone generated by $\Phi_+$.
\end{theo}
Theorem \ref{sg-bary-thm} was first proved by Delcroix \cite{Del2} for smooth Fano compactifications. Later in \cite{LZZ}, Li-Zhou-Zhu gave a generalization of Theorem \ref{sg-bary-thm} for general smooth polarized compactifications. The approach of \cite{LZZ} is to study the properness of Mabuchi K-energy. In \cite{LZ}, Li-Zhou studied the properness of (modified) Ding functional on smooth Fano compactifications.
In \cite{LTZ2}, Li-Tian-Zhu generalized the estimate in \cite{LZZ,LZ} and proved that \eqref{sg-bary} is equivalent to the properness of Ding functional (modulo group action) on the $\mathcal E^1_{K\times K}(M,K_M^{-1})$-space introduced by \cite{BBEGZ, Coman-Guedj-Sahin-Zeriahi}. Hence get the existence result by using the variation method. We highlight that the proof of the necessity of \eqref{sg-bary} does not require the ``fine" assumption of $P$. This can be done by using \cite[Theorem 1.1]{Berman-inve} and a formula of generalized Futaki invariant in \cite[Theorem 3.3]{AK}. In Lemma \ref{nece-sg-bary}, we will show this in detail.

Theorem \ref{sg-bary-thm} gives an explicit and practical way to test existence of (singular) K\"ahler-Einstein metric. One of the main results of \cite{LTZ2} is that the limit of K\"ahler-Ricci flow, starting from a smooth K-unstable $SO_4(\mathbb C)$-compactification is no longer an $SO_4(\mathbb C)$-compactification. On one direction, the famous Hamilton-Tian conjecture (cf. \cite{Ti97, Bam, CW}) suggests that the limit should be a $\mathbb Q$-Fano variety with a singular K\"ahler-Ricci soliton of the same volume as that of initial metric. On the other hand, \cite[Theorem 1.3]{LTZ2} shows that there is no $SO_4(\mathbb C)$-compactification, admitting (singular) K\"ahler-Einstein metric, has the prescribed volume at the same time. There are of course infinitely many $\mathbb Q$-Fano $SO_4(\mathbb C)$-compactifications. In \cite[Section 7.3]{LTZ2}, the volume condition plays a crucial role in reducing the problem to finite cases. In this note, we will refine the estimate of \cite[Section 7.3]{LTZ2} to throw away the volume restriction and completely solve the problem of exhausting $\mathbb Q$-Fano $SO_4(\mathbb C)$-compactifications with K\"ahler-Einstein metrics. Indeed, we can prove a more general result:

\begin{theo}\label{finiteness-thm}
For any semisimple $G$, there are at most finitely many $\mathbb Q$-Fano $G$-compactifications, 
that admits (singular) K\"ahler-Einstein metrics.
\end{theo}
\begin{rem}
Theorem \ref{finiteness-thm} can not be true for a general reductive group. For example, let $G$ be a $2$-dimensional complex torus. Consider the toric surfaces $M_{p,q}$ whose polytope is
$$P(p,q)=\{(x,y)|1-(p|x|+q|y|)\geq0\},$$
where $(p,q)$ is a prime vector with $p,q\geq0$. Obviously by Theorem \ref{sg-bary-thm} each $M_{p,q}$ admits a toric (singular) K\"ahler-Einstein metric.
\end{rem}

We will prove Theorem \ref{finiteness-thm} in Sections 3-4. Consider the case of $G=SO_4(\mathbb C)$. By using Propositions \ref{I(p)-bound} and \ref{for-large-I}, we can reduce the problem of finding $\mathbb Q$-Fano $G$-compactification with K\"ahler-Einstein metrics to a problem of finite cases. Note that in this case rank$(G)=2$, $P$ is always fine. Hence we can further use Theorem \ref{sg-bary-thm} to test the existence.
More precisely, we have the following result, which improves \cite[Theorem 1.3]{LZZ}:
\begin{theo}\label{SO-finiteness-thm}
There are only two $\mathbb Q$-Fano  $SO_4(\mathbb C)$-compactifications,  admitting (singular) K\"ahler-Einstein metrics. Namely, Cases 5.1 and 5.2 given in Section 5.
\end{theo}

\subsection*{Acknowledgement} We would like to thank Professor Xiaohua Zhu for helpful comments which improve this paper a lot.

\section{Preliminary}

\subsection{Polytopes of $\mathbb Q$-Fano group compactifications}
Let $M$ be a $\mathbb Q$-Fano compactification of $G$ with $Z$ the closure of a maximal complex torus $T^{\mathbb C}$ as before. 
We  first  characterize the associated polytope $P$ of $(M,K^{-1}_M)$.  Let $\{F_A\}_{A=1,...,d_0}$ be the facets of $P$ and $\{F_A\}_{A=1,...,d_+}$ be those whose interior intersects $\mathfrak a_+^*$. Suppose that
\begin{align}\label{sg-polytope-eq}
P=\cap_{A=1}^{d_0}\{y\in\mathfrak a^*|l_A^o(y):=\lambda_{A}-u_A(y)\geq0\}
\end{align}
for some prime vector $u_A\in\mathfrak N$, the lattice of one-parameter subgroups, and the facet
$F_A\subseteq\{l^o_{A}=0\},A=1,...,d_0.$ By the $W$-invariance, for each $A\in\{1,...,d_0\}$, there is some $w_A\in W$ such that $w_A(F_A)\in\{F_B\}_{B=1,...,d_+}$. Denote by $\rho_A\,=\,w_A^{-1}(\rho)$. Then the number $\rho_A(u_A)=\frac12\sum_{\alpha\in\Phi_+}|\alpha(u_A)|$ is independent of the choice of $w_A\in W$ and hence is well-defined.

The following is due to \cite[Section 3]{Brion89}.

\begin{lem}\label{sg-polytope-coefficient}
Let $M$ be a $\mathbb Q$-Fano compactification of $G$ with $P$ being the associated polytope. Then for each $A=1,...,d_0$, it holds
\begin{align}\label{sg-relation-lambda}\lambda_{A}\,=\,1+2\rho_A(u_A).
\end{align}
\end{lem}

\begin{proof}
Denote by $B^+$ the (positive) Borel subgroup of $G$ corresponding to $(T^\mathbb C,\Phi_+)$ and $B^-$ be the opposite one. Suppose that $-mK_M$ is a Cartier divisor for some  $m\in\mathbb N$. By \cite[Section 3]{Brion89},  there exists a $B^+\times B^-$-semi-invariant section of $-K_M$,
$$-mK_M=m(\sum_{A'}X_{A'}+2\sum_{\alpha_i\in\Phi_{+,s}}Y_{\alpha_i}),$$
where $\{X_{A'}\}$ is the set of $G\times G$-invariant prime divisors and $Y_{\alpha_i}$ is the prime $B^+\times B^-$-semi-invariant divisor with weight $\alpha_i$ in $\Phi_{+,s}$, the set of simple roots in $\Phi_+$. Note that the corresponding $B^+\times B^-$-weight of this divisor is $2\rho$ (cf. \cite[Section 3.2.4]{Del3} and \cite[Section 7]{Timashev-Sbo}). Thus by adding the divisor of a $B^+\times B^-$-semi-invariant rational function $f$ with weight $-2\rho$, we get a $G\times G$-invariant divisor
$-K_M+\text{div}(f).$ On the other hand, by \cite[Theorem 2.4]{AK}, the prime $G\times G$-invariant divisors of $M$ are in bijections with $W$-orbits of prime toric divisors of $Z$.

Hence, we have
$$-mK_M|_Z\,=\,\sum_Am(1+2\rho_A(u_A))D_A,$$
where $D_A$ is the toric divisor of $Z$ associated to $u_A$. Thus the associated polytope of
 $(Z,-mK_M|_Z)$ is given by
$$P(Z,-mK_M|_Z)\,=\,\cap_{A=1}^{d_0}\{m(1+2\rho_A(u_A))-u_A(y)\geq0\},$$
which is precisely $mP$.  Thus (\ref{sg-relation-lambda}) is true.
\end{proof}

\subsection{Singular K\"ahler-Einstein metric}
For a $\mathbb Q$-Fano variety $M$, by Kodaira's embedding Theorem, there is an integer $\ell > 0$ such that we can embed $M$ into a projective space $\mathbb {CP}^N$ by a basis of $H^0(M, K_M^{-\ell})$, for simplicity, we assume $M\subset \mathbb {CP}^N$. Then we have a metric
$\omega_0\,=\,\frac{1}{\ell}\,\omega_{FS}|_{M}\,\in \,2\pi c_1(M),$
where $\omega_{FS}$ is the Fubini-Study metric of $\mathbb {CP}^N$.
Moreover, there is a Ricci potential $h_0$ of $\omega_0$ such that
$${\rm Ric}(\omega_0)-\omega_0\,=\,\sqrt{-1}\partial\bar\partial h_0, ~{\rm on} ~M_{\rm reg}.$$In the case that $M$ has only
klt-singularities, $e^{h_0}$ is $L^p$-integrate for some $p>1$ (cf. \cite{DT, BBEGZ}).
For a general (possibly unbounded) K\"ahler potential $\varphi$,  we define its complex Monge-Amp\`ere measure $\omega_{\varphi}^n$ by
$$\omega_{\varphi}^n\,=\,\lim_{j\to \infty}\,\omega_{\varphi_j}^n,$$
where $\varphi_j\,=\,{\rm max}\{\varphi,-j\}$.
According to \cite{BBEGZ}, we say that $\varphi$ (or $\omega_{\varphi}^n$)  has  full Monge-Amp\'ere (MA) mass if
$$\int_M \omega_{\varphi}^n\,=\,\int_M \omega_0^n.$$
The MA-measure $\omega_{\varphi}^n$ with full MA-mass has no mass on the pluripolar set of $\varphi$ in $M$. Thus we only need to consider the
measure on  $M_{\rm reg}$.  

\begin{defi}\label{sg-singular-ke} We call $\omega_{\varphi}$ a (singular) K\"ahler-Einstein metric on $M$ with full MA-mass if $\omega_{\varphi}^n$ has full MA-mass and $\varphi$ satisfies the following complex Monge-Amp\'ere equation,
\begin{align}\label{sg-singular-ke-equation}
\omega_{\varphi}^n\,=\,e^{h_0-\varphi}\omega_0^n.
\end{align}
\end{defi}

It has been shown in \cite{BBEGZ} that $\varphi$ is $C^\infty$ on $M_{\rm reg}$ if it is a solution of  (\ref{sg-singular-ke-equation}).  Thus $\omega_{\varphi}$ satisfies the K\"ahler-Einstein equation ${\rm Ric}(\omega_{\varphi})=\omega_{\varphi}$ on $M_{\rm reg}$.

It is showed in \cite[Section 6]{LTZ2} that under the condition \eqref{sg-bary}, the minimizer of the reduced Ding functional on a reduction of the space $\mathcal E^1_{K\times K}(M,K_M^{-1})$ exists and is a solution of \eqref{sg-singular-ke-equation}. See \cite{LTZ2} for details.

\subsection{Necessity of \eqref{sg-bary}}
In this section, we will prove the necessity of \eqref{sg-bary} by testing K-stability and using \cite[Theorem 1.1]{Berman-inve}. We also refer the readers to \cite{Del3}, where the K-stability of a general $\mathbb Q$-Fano spherical variety is discussed in a different framework. Finally, we highlight that the following lemma does not require that the polytope $P$ satisfies the fine condition.
\begin{lem}\label{nece-sg-bary}
Suppose that \begin{align}\label{bar-unstable}\mathbf {b}(P_+)-2\rho\not\in\overline\Xi.\end{align}
Then $M$ is K-unstable. In particular, it can not admit a singular K\"ahler-Einstein metric.
\end{lem}
\begin{proof}
By \cite[Section 2.4]{AK} (see also \cite[Section 4.2]{AB2}), we can associate to each $G\times G$-equivariant test configuration $\mathfrak U$ a unique $W$-invariant, convex piecewise linear function $f_\mathfrak U$ on $P$, and vice versa. And the corresponding generalized Futaki invariant of $\mathfrak U$ is computed by (cf. \cite[Theorem 3.3]{AK} and \cite[Section 3.2]{LZZ})
\begin{align}\label{fut-u}
\text{Fut}(\mathfrak U)=\int_{P_+}\langle y-2\rho,\nabla f_\mathfrak U\rangle\pi\,dy.
\end{align}
We have two cases:\\
\emph{Case-1.} $\mathbf {b}(P_+)-2\rho$ does not in the semisimple part $\mathfrak a_{ss}$ of $\mathfrak a$. Then we take a test configuration $\mathfrak U$ so that $f_\mathfrak U=\xi^iy_i$ for a non zero $\xi\in\mathfrak z(\mathfrak g)$. By \eqref{fut-u},
\begin{align*}
\text{Fut}(\mathfrak U)=\text{Vol}(P_+)\xi(\mathbf{b}(P_+))\not=0.
\end{align*}
Thus $M$ is K-unstable.\\
\emph{Case-2.} $\mathbf {b}(P_+)-2\rho\in\mathfrak a_{ss}$. Let $\{\alpha_i\}_{i=1}^r$ be the simple roots in $\Phi_+$. By \eqref{bar-unstable}, without loss of generality we can assume that
$$\mathbf{b}(P_+)-2\rho=\sum_{i=1}^rc_i\alpha_i,$$
where $c_1<0$. Let $\{\varpi_i\}_{i=1}^r$ the corresponding fundamental weights in $\mathfrak a$ such that $\langle\alpha_i,\varpi_j\rangle=\frac12|\alpha_j|^2\delta_{ij}.$ Put $$f(y)=\max_{w\in W}\{\langle w\cdot \varpi_1, y\rangle\}.$$
Then $f$ is a $W$-invariant, convex piecewise linear function $f_\mathfrak U$ defined on $P$, which is not affine on the whole $P$. Hence it gives a non-product test configuration $\mathfrak U_f$. Note that since $\varpi_1$ is dominant,
$$f|_{P_+}(y)=\langle\varpi_1,y\rangle.$$
By \eqref{fut-u}, we have
\begin{align*}
\text{Fut}(\mathfrak U_f)=\frac12c_1|\alpha_1|^2\text{Vol}(P_+)<0.
\end{align*}
Thus we see that $M$ is K-unstable. The last point then follows from \cite[Theorem 1.1]{Berman-inve}.
\end{proof}

\section{Classification of $\mathbb Q$-Fano $G$-compactifications}
We first prove an elementary lemma:
\begin{lem}\label{linear-alg}
Let $A=(a_{ij}),1\leq i,j\leq r$ be an $r\times r$-positive defined real matrix such that $a_{ij}\leq0$ whenever $i\not=j$. Denote by $A^{-1}=(a^{ij})$ its inverse. Then
\begin{align}\label{ij-entry}
a^{ij}\geq0,~1\leq i,j\leq r.
\end{align}
\end{lem}

\begin{proof}
We claim that there is an upper-triangle matrix $B=(b_{ij})$ with $b_{ij}\geq0$ such that
$$A=(B^T)^{-1}B^{-1}.$$
Once this is proved, the lemma follows directly from $A^{-1}=BB^T$ and $b_{ij}\geq0$.

We prove the claim by induction, put
$$B_1=\left(\begin{aligned}&1&&-\frac{a_{12}}{a_{11}}&&...&&-\frac{a_{1i}}{a_{11}}&&...&&-\frac{a_{1n}}{a_{11}}&\\
&0&&1&&...&&0&&...&&0&\\
&...&&...&&...&&...&&...&\\
&0&&0&&...&&0&&...&&1&\end{aligned}\right).$$
Then $B_1^TAB_1=\text{diag}(a_{11},A')$, where $A'=(a'_{ij}),2\leq i,j\leq r$ is an $(r-1)\times(r-1)$-positive defined matrix with
$$a'_{ij}=a_{ij}-\frac{a_{i1}a_{1j}}{a_{11}}\leq0,~\forall i\not=j.$$
By induction we prove the claim.
\end{proof}
We can conclude that:
\begin{coro}\label{inverse-cartan}
Let $\{\alpha_i\}_{i=1}^r$ be the simple roots in $\Phi_+$ and $\{\varpi_i\}_{i=1}^r$ the corresponding fundamental weights in $\mathfrak a$. Then
$$\langle\varpi_i,\varpi_j\rangle\geq0,~1\leq i,j\leq r.$$
\end{coro}
\begin{proof}
This is equivalent to that all entries of the inverse of the Cartan matrix $C=(c_{ij})$ are non-negative. By definition, $$c_{ij}=2\frac{\langle\alpha_i,\alpha_j\rangle}{|\alpha_i|^2}.$$
It is direct to see that $$C=\Lambda A,$$ where $$\Lambda=\text{diag}(\frac{2}{|\alpha_1|^2},...,\frac{2}{|\alpha_r|^2})$$ is a positive defined diagonal matrix and $A=(\langle\alpha_i,\alpha_j\rangle)$ is a positive defined matrix. Since $\alpha_1,...,\alpha_r$ are simple roots, it holds $$\langle\alpha_i,\alpha_j\rangle\leq0,~\text{ if }i\not=j.$$ The corollary then can be concluded from Lemma \ref{linear-alg}.
\end{proof}

\begin{coro}\label{coe-u}
Let $u=\sum_{i=1}^rc_i\varpi_i$ be a vector in $\overline{\mathfrak a_+}$.
\begin{itemize}
\item[(1)] $\rho(u)=0$ if and only if $u=0$;
\item[(2)] The coefficients
$$c_j\leq\frac{4\rho(u)}{|\alpha_j|^2},~\forall j\in\{1,...,r\}.$$
\end{itemize}
\end{coro}
\begin{proof}
Since $u=\sum_{i=1}^rc_i\varpi_i\in\overline{\mathfrak a_+}$, $c_j\geq0$ for all $j$. By Corollary \ref{inverse-cartan}, we have
$$\rho(u)=\sum_{1\leq i,j\leq r}c_i\langle\varpi_i,\varpi_j\rangle\geq\sum_{1\leq i\leq r}c_i\langle\varpi_i,\varpi_i\rangle.$$
Thus we get (1). For tem (2), since
$\varpi_i(\alpha_j)=\frac12{|\alpha_j|^2}\delta_{ij},$
we have
$$c_j=\frac{2\alpha_j(u)}{|\alpha_j|^2}\leq\frac{2\sum_{\alpha\in\Phi_+}\alpha(u)}{|\alpha_j|^2}=\frac{4\rho(u)}{|\alpha_j|^2}.$$
\end{proof}
Then we introduce a label $I(P)$ to each $P$ for our classification. By Lemma \ref{sg-polytope-coefficient}, each outer facet \footnote {An facet of $P_+$ is called an outer one if it does not lie in any Weyl wall, cf. \cite{LZZ}.} of $P_+$ must lies on some line
\begin{align}\label{sg-line-pq}l_{A}(y)=(1+2\rho(u_A))-u_A(y)\,=\,0
\end{align}
for some prime norm $u\in\mathfrak N$. Assume that each $l_{A}\geq 0$ on $P$. By convexity and $W$-invariance of $P$, we get $u\in\overline{\mathfrak a_+}.$ Consider the intersection of $P_+$ with the ray $\{t\rho|t\geq0\}$, namely, a point $t_0\rho\in\partial P_+, t_0>0$. Then
$$t_0=2(1+\frac1{2\rho(u_{A_0})})$$ for some $A_0\in\{1,...,d_+\}$, and there is a corresponding outer facet $F_{A_0}$ of $P_+$ which lies on some $\{l_{A_0}=0\}$.

We associate this number $$I(P):=2\rho(u_0)$$ to each $\mathbb Q$-Fano polytope $P$ (and hence each $\mathbb Q$-Fano $G$-compactifications). Now we look at the intersection of the ray $t\rho$ with other hyperplanes $\{y|l_A=0\}$, where $A\in\{1,...,d_+\}$. By Corollary \ref{coe-u} (1), for any $A$, the intersection point $t_A\rho$, where
$$t_A=2(1+\frac1{2\rho(u_A)})>0$$
always exists.

On the other hand, by convexity, for other $A\in\{1,...,d_+\}$, if $l_A(t_A\rho)=0$, it must hold
$t_A\geq t_0,$
or equivalently, $$\rho(u_A)\leq\rho(u_{A_0})=I(P).$$
Thus, for each $u_A=\sum_ic^A_i\varpi_i\in\mathfrak N\subset\text{Span}_{\mathbb Z}\{\varpi_i|i=1,...,r\}$, by Corollary \ref{coe-u} (2),
$$c^A_j\in\mathbb Z\cap[0,\frac{4I(P)}{|\alpha_j|^2}].$$
Hence, for fixed $I(P)$, there are only finite choices of $u_A$ in \eqref{sg-polytope-eq}.
We conclude that:
\begin{prop}\label{I(p)-bound}
For each $k\in\mathbb N_+$, there are only finitely possible $\mathbb Q$-Fano $G$-compactifications whose polytope $P$ satisfies $I(P)=k$.
\end{prop}

\section{Estimate of barycenter}

We will give an estimate of $\mathbf {b}(P_+)$.
For a domain $\Omega\subset\mathbb R^r$ and a function $f:\Omega\to\mathbb R$, denote by
$$\bar f_\pi(\Omega):=\frac{\int_\Omega f\pi dy}{\int_\Omega \pi dy}$$
the average of $f$ on $\Omega$ with respect to the weight $\pi$.

\begin{prop}\label{for-large-I}
Let
$\varphi(y):=u_{A_0}(y).$
Then this a number $\omega(n)>0$ which depends only on $n$ such that if $I(P)\geq\omega(n)$,
\begin{align}\label{lem-21}
\varphi(\mathbf{b}(P_+))<\varphi(2\rho).
\end{align}
In particular, the corresponding $G$-compactification does not admit any K\"ahler-Einstein metric.
\end{prop}

\begin{proof}
Consider the hyperplane
$$\hat\Pi_0:=\{\hat l_{A_{0}}(y):=2\rho(u_{A_0})-u_{A_0}(y)=0\}.$$
Then $P_+$ cuts out a codimensional one bounded polytope $\hat P=\hat\Pi_0\cap P_+$ on $\hat\Pi_0$. We can divide $P_+$ into three parts:

\begin{align*}
\Omega_1&=\{t\hat P|t\in[0,1]\};\\
\Omega_2&=P_+\cap\{y|\varphi(y)\geq 2\rho(u_{A_0})\};\\
\Omega_3&=P_+\setminus (\Omega_1\cup\Omega_2).
\end{align*}

Take a parametrization
$$y(t,s^1,...,s^{r-1})=t\hat y(s^1,...,s^{r-1}),$$
where $\hat y(s^1,...,s^{r-1})$ is a parametrization of $\hat \Pi_0$. By direct computation, we see that for a homogenous function $f_m(y):=\prod_{i=1}^r(y_i)^{m_i}$,
$$f_m(y) dy=t^{r-1+\sum_{i=1}^rm_i}f_m(\hat y)dt\wedge d\hat y.$$
Hence
\begin{align}\label{V-2}
\text{Vol}(\Omega_1)=\frac1n\int_{\hat P}\pi(\hat y)d\hat y,
\end{align}
and
\begin{align}\label{phi-2}
\int_{\Omega_1}\varphi \pi dy=\frac{2\rho(u_{A_0})}{n+1}\int_{\hat P}\pi(\hat y)d\hat y.
\end{align}
On the other hand, by convexity, we have
$$\Omega_2\subset\{t\hat P|t\in[1,1+\frac1{2\rho(u_{A_0})}]\}.$$
Thus
\begin{align}\label{V-4}
\text{Vol}(\Omega_2)\leq\frac1n\left((1+\frac1{2\rho(u_{A_0})})^n-1\right)\int_{\hat P}\pi(\hat y)d\hat y=:V_2'.
\end{align}
Obviously,
\begin{align}\label{phi-4}
\varphi(y)\leq2\rho(u_{A_0})+1\text{ on }\Omega_2.
\end{align}
Combining \eqref{V-2}-\eqref{phi-4}, we see that the average of $\varphi$ on $\Omega_2\cup\Omega_4$ satisfies
\begin{align*}
\bar{\varphi}_\pi(\Omega_1\cup\Omega_2)&=\frac1{\text{Vol}(\Omega_1\cup\Omega_2)}{\int_{\Omega_1\cup\Omega_2}} \varphi \pi  dy'\\&\leq\frac{(\int_{\Omega_1}\varphi \pi(y) d y)+2\rho(u_{A_0})(1+\frac1{2\rho(u_{A_0})})V_2'}{\text{Vol}(\Omega_1)+V_1'}\\
&=2\rho(u_{A_0})\cdot\frac{\frac n{n+1}+(1+\frac1{2\rho(u_{A_0})})((1+\frac1{2\rho(u_{A_0})})^n-1)}{(1+\frac1{2\rho(u_{A_0})})^n}.
\end{align*}

Hence there is a number $\omega(n)>0$ which depends only on $n$ such that
$$\bar{\varphi}_\pi (\Omega_1\cup\Omega_2)< 2\rho(u_{A_0})=\varphi(2\rho),~I(P)\geq\omega(n).$$
Note that
$$\varphi(y)<\varphi(2\rho)\text{ on }\Omega_3.$$
By the fact that
$$\varphi(\mathbf {b}(P_+))=\bar \varphi_\pi (P_+)$$
and the above two inequalities, we get the desired estimate \eqref{lem-21}.

Since
$$2\rho+\overline{\Xi}\subset\{\varphi(y)\geq\varphi(2\rho)\}.$$
The last statement of the proposition then follows from Lemma \ref{nece-sg-bary}.
\end{proof}

\begin{proof}[Proof of Theorem \ref{finiteness-thm}]
By Proposition \ref{lem-21}, it suffices to consider the cases with $I(P)\leq\omega(n)$.  While by Proposition \ref{I(p)-bound}, when $I(P)\leq\omega(n)$ there are only finitely many possible choices of $P$, which may admit (singular) K\"ahler-Einstein metric. Hence we prove the theorem.
\end{proof}

\section{Application to $\mathbb Q$-Fano  $SO_4(\mathbb C)$-compactifications}
\subsection{The classification result}
In this section, we show Theorem \ref{SO-finiteness-thm} as an application of Theorem \ref{finiteness-thm}. We adopt the notations as in \cite[Section 7]{LTZ2}. In particular $p_0$ in \cite[Section 7.2]{LTZ2} is precisely $I(P)$ defined here in Section 3. By direct computation, we see that $\omega(6)=3.83$. Thus, it suffices to check all possible campactifications with $p_0(=I(P))\leq3.$

By using software \texttt{Wolframe Mathematica 11.3}, we find that there are only two $\mathbb Q$-Fano  $SO_4(\mathbb C)$-compactifications which admit K\"ahler-Einstein metrics. The corresponding polytopes $P_+$ are:
\begin{itemize}
\item [Case 5.1.] $P_+=\{(x,y)|0\leq x\leq3,~-x\leq y\leq x\}$;
\item [Case 5.2.] $P_+=\{(x,y)|0\leq x+y\leq3,~0\leq x-y\leq3\}.$
\end{itemize}
\subsection{The \texttt{Wolframe Mathematica} code}
The following is the \texttt{Wolframe Mathema-} \texttt{tica 11.3} code for finding the $\mathbb Q$-Fano $SO_4(\mathbb C)$-compactifications with a given $p_0$ (the function ``FindComp") and test whether the compatification admits K\"ahler-Einstein metrics (the functions ``Coordinatecalculation" ``QKE"). The basic idea of finding compatifications is to add the edges of the polytope corresponding to $p$ as $p$ goes from $p_0$ to $1$, we determine the admissible edges (the set ``Ipos" and ``Ineg") for each $p$ corresponding to the polytope constructed in the $(p_0-p)$-th step, then choosing at most two edges from the admissible edges (one from ``Ipos" and the other from ``Ineg", of course the choice can be empty) and add them to the the set of edges of former constructed polytopes.   We should notice that the there are some redundance in the result of compatifications since under the action of Weyl group, the same polytope may have different forms.
$$
\begin{aligned}
&\texttt{FindComp[p\_] :=}\\
&\texttt{ (Polytopeset = \{\};}\\
&\texttt{  Lneg = \{\};}\\
&\texttt{  Lpos = \{\};}\\
&\texttt{  For[i = 1, i <= p - 1, i++,}\\
&\texttt{   If[CoprimeQ[i, p] == True, Lpos = Append[Lpos, i];}\\
&\texttt{    Lneg = Append[Lneg, -i]]];}\\
&\texttt{  For[i = 1, i <= Length[Lpos], i++,}\\
&\texttt{   Polytopeset = Append[Polytopeset, \{\{p, Lpos[[i]]\}\}];}\\
&\texttt{   For[j = 1, j <= Length[Lneg], j++,}\\
&\texttt{    Polytopeset =}\\
&\texttt{      Append[Polytopeset, \{\{p, Lpos[[i]]\}, \{p, Lneg[[j]]\}\}];}\\
&\texttt{    ];}\\
&\texttt{   ];}\\
&\texttt{  For[j = 1, j <= Length[Lneg], j++,}\\
\end{aligned}
$$
$$
\begin{aligned}
&\texttt{   Polytopeset = Append[Polytopeset, \{\{p, Lneg[[j]]\}\}];}\\
&\texttt{   ];}\\
&\texttt{  For[k = p - 1, k >= 1, k--,}\\
&\texttt{   len = Length[Polytopeset];}\\
&\texttt{   For[l = 1, l <= len, l++,}\\
&\texttt{    Polytope = Polytopeset[[l]];}\\
&\texttt{    pt = Polytope[[1]][[1]];}\\
&\texttt{    qt = Polytope[[1]][[2]];}\\
&\texttt{    pf = Polytope[[Length[Polytope]]][[1]];}\\
&\texttt{    qf = Polytope[[Length[Polytope]]][[2]];}\\
&\texttt{    Lpos = \{\};}\\
&\texttt{    Lneg = \{\};}\\
\end{aligned}
$$
$$
\begin{aligned}
&\texttt{    If[Length[Polytope] == 1,}\\
&\texttt{     For[i = k, i >= (2 k qt + pt + qt - k)/(1 + 2 pt), i--,}\\
&\texttt{      If[CoprimeQ[i, k] == True, Lpos = Append[Lpos, i];}\\
&\texttt{        ];}\\
&\texttt{      ];}\\
&\texttt{     For[j = -k, j <= (2 k qt + qt - pt + k)/(1 + 2 pt), j++,}\\
&\texttt{      If[CoprimeQ[j, k] == True, Lneg = Append[Lneg, j];}\\
&\texttt{        ];}\\
&\texttt{      ];}\\
&\texttt{     ];}\\
&\texttt{    If[Length[Polytope] >= 2,}\\
&\texttt{     pt2 = Polytope[[2]][[1]];}\\
&\texttt{     qt2 = Polytope[[2]][[2]];}\\
&\texttt{     pf2 = Polytope[[Length[Polytope] - 1]][[1]];}\\
&\texttt{     qf2 = Polytope[[Length[Polytope] - 1]][[2]];}\\
&\texttt{     For[i = k, i >= (2 k qt + pt + qt - k)/(1 + 2 pt), i--,}\\
\end{aligned}
$$
$$
\begin{aligned}
&\texttt{If[CoprimeQ[i, k] == True \text{\&\&}}\\
&\texttt{         i (pt - pt2) >= (k qt - pt2 qt - k qt2 + pt qt2),}\\
&\texttt{        Lpos = Append[Lpos, i];}\\
&\texttt{        ];}\\
&\texttt{      ];}\\
\end{aligned}
$$
$$
\begin{aligned}
&\texttt{     For[j = -k, j <= (2 k qf + qf - pf + k)/(1 + 2 pf), j++,}\\
&\texttt{      If[CoprimeQ[j, k] == True \text{\&\&}}\\
&\texttt{         j (pf - pf2) <= (k pf - pf2 qf - k qf2 + pf qf2),}\\
&\texttt{        Lneg = Append[Lneg, j];}\\
&\texttt{        ];}\\
&\texttt{      ];}\\
&\texttt{     ];}\\
\end{aligned}
$$
$$
\begin{aligned}
&\texttt{    For[i = 1, i <= Length[Lpos], i++,}\\
&\texttt{     Polytopeset =}\\
&\texttt{      Append[Polytopeset, Prepend[Polytope, \{k, Lpos[[i]]\}]];}\\
&\texttt{     For[j = 1, j <= Length[Lneg], j++,}\\
&\texttt{      Polytopeset =}\\
&\texttt{        Append[Polytopeset,}\\
&\texttt{         Append[Prepend[Polytope,\{k, Lpos[[i]]\}], \{k, Lneg[[j]]\}]];}\\
&\texttt{      ];}\\
&\texttt{     ];}\\
\end{aligned}
$$
$$
\begin{aligned}
&\texttt{    For[j = 1, j <= Length[Lneg], j++,}\\
&\texttt{     Polytopeset =}\\
&\texttt{       Append[Polytopeset, Append[Polytope, \{k, Lneg[[j]]\}]];}\\
&\texttt{     ];}\\
&\texttt{    ];}\\
&\texttt{   ];}\\
&\texttt{  Return[Polytopeset];}\\
&\texttt{  )}\\
&\texttt{  Coordinatecalculation[Polyset\_] :=}\\
&\texttt{ (Coordinateset = \{\};}\\
&\texttt{  For[i = 1, i <= Length[Polyset], i++,}\\
&\texttt{   Polyt = Polyset[[i]];}\\
&\texttt{   Vol = NIntegrate[}\\
&\texttt{     x\^2 y\^2 Boole[x - y > 0] Boole[x + y > 0] Product[}\\
&\texttt{       Boole[2 Polyt[[s]][[1]] + 1 >}\\
&\texttt{         Polyt[[s]][[1]] x + Polyt[[s]][[2]] y], \{s, 1,}\\
&\texttt{        Length[Polyt]\}], \{x, 0, Infinity\}, \{y, -Infinity, Infinity\}];}\\
\end{aligned}
$$
$$
\begin{aligned}
&\texttt{   xcord =}\\
&\texttt{    NIntegrate[}\\
&\texttt{     x\^3 y\^2 Boole[x - y > 0] Boole[x + y > 0] Product[}\\
&\texttt{       Boole[2 Polyt[[s]][[1]] + 1 >}\\
&\texttt{         Polyt[[s]][[1]] x + Polyt[[s]][[2]] y], \{s, 1,}\\
&\texttt{        Length[Polyt]\}], \{x, 0, Infinity\}, \{y, -Infinity, Infinity\}];}\\
&\texttt{   ycord =}\\
&\texttt{    NIntegrate[}\\
&\texttt{     x\^2 y\^3 Boole[x - y > 0] Boole[x + y > 0] Product[}\\
&\texttt{       Boole[2 Polyt[[s]][[1]] + 1 >}\\
&\texttt{         Polyt[[s]][[1]] x + Polyt[[s]][[2]] y], \{s, 1,}\\
&\texttt{        Length[Polyt]\}], \{x, 0, Infinity\}, \{y, -Infinity, Infinity\}];}\\
&\texttt{   Coordinateset = Append[Coordinateset, \{xcord/Vol, ycord/Vol\}];}\\
&\texttt{   ];}\\
&\texttt{  Return[Coordinateset];}\\
&\texttt{  )}\\
\end{aligned}
$$
$$
\begin{aligned}
&\texttt{  QKE[Cordset\_] :=}\\
&\texttt{ (KE = \{\};}\\
&\texttt{  For[i = 1, i <= Length[Cordset], i++,}\\
&\texttt{   cord = Cordset[[i]];}\\
&\texttt{   KE = Append[KE,}\\
&\texttt{     cord[[1]] + cord[[2]] >= 2 \text{\&\&} cord[[1]] - cord[[2]] >= 2];}\\
&\texttt{   ];}\\
&\texttt{  Return[KE];}\\
&\texttt{  )}\\
\end{aligned}
$$


\clearpage

\begin{thebibliography}{10}
\bibitem{AB1}V. A. Alexeev and M. Brion, \textit{Stable reductive varieties \uppercase\expandafter{\romannumeral 1}: Affine varieties}, Invent. Math., \textbf{157} (2004), 227-274.

\bibitem{AB2} V. A. Alexeev,  and M. Brion, \textit{Stable reductive varieties \uppercase\expandafter{\romannumeral 2}: Projective case}, Adv. Math., \textbf{184} (2004), 382-408.

\bibitem{AK} V. A. Alexeev and  L. V. Katzarkov,  \textit{On K-stability of reductive varieties}, Geom. Funct. Anal., \textbf{15} (2005), 297-310.

\bibitem {Bam} R. Bamler, \textit{Convergence of Ricci flows with bounded scalar curvature},  Ann. Math., \textbf {188}  (2018), 753-831.

\bibitem{Berman-inve} R. Berman,
\textit{K-stability of $\mathbb Q$-Fano varieties admitting K¡§ahler-Einstein metrics},  Invent. Math. \textbf{203} (2015), 973-1025.

\bibitem{BBEGZ} R. Berman, S. Boucksom, P. Eyssidieux, V. Guedj and A. Zeriahi,
\textit{K\"ahler-Einstein metrics and the K\"ahler-Ricci flow on log Fano varieties}, arXiv:1111.7158v3,
to appear in J. Reine Angew. Math.

\bibitem{Brion89}M. Brion,  \textit{Groupe de Picard et nombres caract\'eristiques des vari\'et\'es sph\'eriques}, Duke. Math. J., \textbf{58} (1989), 397-424.

\bibitem{Coman-Guedj-Sahin-Zeriahi} D. Coman, V. Guedj, S. Sahin and A. Zeriahi, \textit{Toric pluripotential theory}, arXiv:1804.03387.

\bibitem{CW}X. Chen, X. and B. Wang,
\textit{Space of Ricci flows (\uppercase\expandafter{\romannumeral 2})}, arXiv:1405.6797.

\bibitem{Do}S. Donaldson, \textit{Interior estimates for solutions of Abreu's equation}, Collect. Math., \textbf{56} (2005), 103-142.

\bibitem{Del2}T. Delcroix, \textit{K\"ahler-Einstein metrics on group compactifications}, Geom. Func. Anal., \textbf{27} (2017), 78-129.

\bibitem{Del3}T. Delcroix, \textit{K-Stability of Fano spherical varieties}, arXiv:1608.01852.

\bibitem{DT} W. Ding and G. Tian,   \textit{K\"ahler-Einstein metrics and the generalized
Futaki invariants},  Invent. Math.,  \textbf{110} (1992), 315-335.

\bibitem{LTW17} C. Li, G. Tian and F. Wang,  \textit{On Yau-Tian-Donaldson conjecture for singular Fano varieties}, arXiv:1711.09530

\bibitem{LZZ} Y.  Li,  B.  Zhou  and   X.H.   Zhu, \textit{K-energy on polarized compactifications of Lie
groups},  J. of  Func. Analysis, \textbf{275} (2018), 1023-1072.

\bibitem{LZ}Y. Li and B. Zhou, \textit{Mabuchi metrics and properness of modified Ding functional},  Pacific J. Math., \textbf{302} (2019), 659-692.

\bibitem{LTZ2} Y.  Li, G. Tian and  X.H.   Zhu,  \textit{Singular K\"ahler-Einstein metrics on $\mathbb Q$-Fano compactifications of Lie groups},  arXiv:2001.11320.

\bibitem {Ti97} G. Tian,
\textit{K\"ahler-Einstein metrics with positive scalar curvature},
Invent. Math., \textbf{130} (1997), 1-37.

\bibitem{Timashev-Sbo}
\font\fontWCA=wncyr8 {\fontWCA D. A. Timash\oe v}, \textit{\font\fontWCA=wncyr8 {\fontWCA \char'003kvivariantnye kompaktifikatsii reduktivnykh grupp}}, \font\fontWCA=wncyr8 {\fontWCA Matematicheski\ae\,  Sbornik}, \textbf{194} (2003), 119-146. \\
Eng.: D. A. Timash\"ev, \textit{Equivariant compactification of reductive groups}, Matematicheskij Sbornik, \textbf{194} (2003), 119-146.
\end{thebibliography}
\end{document}